\documentclass[12pt]{amsart}
\usepackage[margin=1in,includeheadfoot]{geometry}
\usepackage{amsmath}
\usepackage{slashed}
\usepackage{graphicx}

\usepackage[margin=1in,includeheadfoot]{geometry}
\usepackage{amsmath}
\usepackage{slashed}
\usepackage{graphicx}
\usepackage{mathrsfs}
\usepackage{txfonts}

\usepackage{amssymb,dsfont}
\usepackage{enumerate}
\usepackage{slashed}
\usepackage{fancyhdr}
\usepackage{aliascnt}
\usepackage{pdfsync}
\usepackage{hyperref}

\newtheorem{thm}{Theorem}[section]

\newtheorem{lem}[thm]{Lemma}
\newtheorem{prop}[thm]{Proposition}
\theoremstyle{definition}

\theoremstyle{remark}

\numberwithin{equation}{section}

\newcounter{stepnum}

\def\bee{\begin{eqnarray}}
	\def\beee{\begin{eqnarray*}}
		\def\eee{\end{eqnarray}}
	\def\eeee{\end{eqnarray*}}

\def\ba{\begin{array}}
	\def\ea{\end{array}}

\def\R{\mathbb R}

\makeatletter

\newcommand{\Rmnum}[1]{\expandafter\@slowromancap\romannumeral #1@}
\makeatother

\begin{document}
	
	\title[Super-Liouville equation with a Dirichlet boundary]{Super-Liouville equation with a Dirichlet-chiral boundary}

	\author[Liu]{Lei Liu}
	\address{School of Mathematics and Statistics, Key Laboratory of Nonlinear Analysis and Applications (Ministry of Education), Hubei Key Laboratory of Mathematical Sciences, Central China Normal
		University, Wuhan, 430079, People's Republic of China}%
	\email{leiliu2020@ccnu.edu.cn}
	
	\author[Zhu]{Mingyuan Zhu}
	\address{School of Mathematics and Statistics,  Central China Normal
		University, Wuhan, 430079, People's Republic of China}
	\email{2056396052@qq.com}

	\thanks{}

	\subjclass[2010]{}
	
	\date{\today}
	\begin{abstract}
		In this paper, we investigate the compactness problem for the super-Liouville equation under Dirichlet-chiral boundary conditions. We prove that no blow-up occurs in this setting.
	\end{abstract}
	\maketitle
	\section{Introduction}
	
The compactness of the solution space of a nonlinear partial differential equation plays an important role in studying the existence of solutions. For many interesting geometric PDEs, understanding their solution spaces remains a challenging problem. A powerful tool in this context is the so-called blow-up analysis, which provides valuable insight precisely when the solution space lacks compactness.

The classical Liouville equation for a real-valued function $u$ on a Riemann surface $(M,g)$ is given by
\begin{equation}\label{equ:17}
-\Delta_g u=2e^{2u}-K_g,\end{equation} which plays a fundamental role in many two dimensional physical models, complex analysis and differential geometry of Riemann surface, in particular in the problem of prescribed Gaussian curvature. The interior blow-up analysis for Liouville equation was  systematically developed by Brezis-Merle \cite{Brezis},  Li-Shafrir \cite{Li-Shafrir} and Li \cite{Li}, among others.

For the blow-up analysis of Liouville type equation near the Neumann boundary, Guo-Liu \cite{Guo-Liu} proved Brezis-Merle type concentration phenomenon and Li-Shafrir type quantization property. Later, Bao-Wang-Zhou \cite{Bao-Wang-Zhou} extended  Li's work \cite{Li} to the boundary case.  For Dirichlet boundary condition, Brezis-Merle \cite{Brezis} proved that $u_n$ is locally uniformly bounded in the	interior domain of $M$. In particular, blow-up cannot occur at an interior point. Later, Nagasaki-Suzuki \cite{Nagasaki-Suzuki-1990} obtained a uniform	boundary estimate which implies that the boundary blow-up is also	excluded.

In physics, Liouville equation also occurs naturally in string theory as discovered by Polyakov \cite{Polyakov}, from the gauge anomaly in quantizing the string action. There is also a natural supersymmetric version of the Liouville functional, coupling the bosonic scalar field to a fermionic spinor field. Motivated from supersymmetric string theory, Jost-Wang-Zhou \cite{Jost-Wang-Zhou} introduced the following super-Liouville equation for a real-valued function $u$ and a complex valued spinor $\psi$
\begin{align}\label{equ:32}
\begin{cases}
-\Delta_g u&=2e^{2u}-e^u\langle \psi,\psi\rangle-K_g,\\
-\slashed{D}_g\psi &=e^u\psi,
\end{cases} \ \ in \ \ M.
\end{align}
Here, $(M,g)$ is a Riemann surface with a fixed spin structure, $\Sigma M$  the spinor bundle over $M$ and $\langle\cdot,\cdot\rangle_{\Sigma M}$  the metric on $\Sigma M$. Choosing a local orthonormal basis ${e_\alpha,\alpha=1,2}$ on $M$, the  Dirac operator is
defined as $\slashed{D}_g:=e_\alpha\cdot\nabla_{e_\alpha}$, where $\nabla$ is the spin connection on $\Sigma M$ and $\cdot$ is the Clifford multiplication. This multiplication is  skew-adjoint:
\[
\langle X\cdot\psi,\varphi \rangle_{\Sigma M}=-\langle \psi,X\cdot\varphi \rangle_{\Sigma M}
\]
for any $X\in\Gamma(TM)$, $\psi,\ \varphi\in\Gamma(\Sigma M)$.

Similarly to Liouville equation, super-Liouville equation is also conformally invariant, which implies that the solution space is in general not compact. The blow-up analysis for super-Liouville equation in the interior case and Neumann boundary case were systematically developed in \cite{Jost-Wang-Zhou,Jost-Wang-Zhou-Zhu-1, Jost-Zhou-Zhu-2014} etc.

	In this paper, we study the blow-up analysis for the super-Liouville equation near a Dirichlet-chiral boundary. In particular, we show that the solution space of the Dirichlet-chiral boundary problem is compact. That is, no blow-up occurs in this setting.

For simplicity of notations, we assume $M=\Omega\subset\mathbb R^2$ is a bounded smooth domain, $u_0$ is a
	fixed smooth function in $\overline{\Omega}$. Let $(u_n,\psi_n)$ be a sequence of smooth solutions to
	\begin{equation}\label{equat:01}
		\begin{cases}
			-\Delta u_n= 2e^{2u_n} - e^{u_n} |\psi_n|^2 \ \ &in\ \  \Omega,\\
			-\slashed{D}\psi_n= e^{u_n}\psi_n \ \ &in\ \ \ \Omega,\\
			u_n=u_0  \ \ &on\ \  \partial  \Omega,\\
			\mathbf{B}\psi_n=0 \ \ &on\ \  \partial  \Omega,
		\end{cases}
	\end{equation}where $$\mathbf{B}\psi_n:=\frac{1}{2}(I-ie_1\cdot)\psi_n=\frac{1}{2}(\psi_n -ie_1\cdot\psi_n)$$ is the chiral boundary and $e_1$ is the tangent vector field of $\partial \Omega$. For more details on chiral boundary condition, we refer to \cite{Gibbons,Hijazi}.

	Define the blow-up set
	\begin{align*}
		\Sigma_1&:=\{x\in \overline{\Omega}\ |\ \mbox{there is a sequence of points } x_n\to x\mbox{ such that } \limsup_{n\to\infty}u_n(x_n) =+\infty\},\\
		\Sigma_2&:=\{x\in \overline{\Omega}\ |\ \mbox{there is a sequence of points } x_n\to x\mbox{ such that } \limsup_{n\to\infty}|\psi_n(x_n)| =+\infty\}.
	\end{align*}
Denote
	$$
	\Sigma=\Sigma_1\cup\Sigma_2
	$$ and
$$
	E(u_n,\psi_n;\Omega)
	:=
	\int_\Omega
	\left(
	e^{2u_n}+|\psi_n|^4
	\right)dx.
	$$
	
\

	Our main result is the following.
	
	\begin{thm}\label{thm:main}
		Let $(u_n,\psi_n)$ be a sequence of smooth solutions of
		\eqref{equat:01} with finite energy
		$$
		E(u_n,\psi_n;\Omega)\leq C.
		$$
		Then $		\Sigma=\emptyset.$
	\end{thm}

\

Compared to the Liouville equation, the presence of an additional spinor term in the model means that the term $2e^{2u_n}-e^{u_n}|\psi_n|^2$ does not have a fixed sign. Consequently, the maximum principle and the boundary estimates used for the classical Liouville equation cannot be applied directly. To prove Theorem \ref{thm:main}, we firstly  apply the interior blow-up theory for the super-Liouville equation developed in \cite{Jost-Wang-Zhou-Zhu-1} and show that no interior blow-up points exist. See Lemma \ref{lem:compactness-away-blowup}.

Second, to rule out boundary blow-up, we establish the following Pohozaev-type identity near a Dirichlet-chiral boundary (see Lemma \ref{lem:boundary-pohozaev-identity})
\begin{align}\label{equat:06}
			&\int_{B_r^+}
			\left(
			2e^{2 u_n}
			-
			e^{ u_n}|\psi_n|^2
			\right)\,dx
			\notag\\
			&\quad =
			r\int_{\partial^+B_r^+}
			\left(
			\left|\frac{\partial u_n}{\partial\nu}\right|^2
			-\frac12|\nabla  u_n|^2
			+e^{2  u_n}
			\right)\,d\sigma
			-\frac14
			\int_{\partial^+B_r^+}
			\left\langle
			\frac{\partial \psi_n}{\partial\nu},
			(x+\bar x)\cdot \psi_n
			\right\rangle\,d\sigma
			\notag\\
			&\qquad
			-\frac14
			\int_{\partial^+B_r^+}
			\left\langle
			(x+\bar x)\cdot \psi_n,
			\frac{\partial \psi_n}{\partial\nu}
			\right\rangle\,d\sigma
			+\int_{\partial^0B_r^+}
			s\,\frac{\partial u_0}{\partial s}(s,0)\,
			\frac{\partial  u_n}{\partial\nu}(s,0)\,ds .
		\end{align}
Then, on one hand, using potential analysis, we show that the right-hand side of \eqref{equat:06} tends to zero along a sequence of special slices $r_j\to 0$ (see Lemma \ref{lem:boundary-decay-estimates} and Lemma \ref{lem:boundary-radial-gradient-vanishing}). On the other hand, we prove that if blow-up occurs, then the concentrated energy must be greater than $4\pi$ (where $4\pi$ is exactly the energy of a single bubble) i.e.
\begin{equation}\label{equat:07}
\liminf_{n\to\infty}\int_{B_r^+}
			\left(
			2e^{2 u_n}
			-
			e^{ u_n}|\psi_n|^2
			\right)\,dx\geq 4\pi,\ \ \forall\ r>0.\end{equation} Then we get a contradiction.

For \eqref{equat:07}, using the small-energy regularity lemma (Lemma \ref{lem:small-energy-regul}), we first show that all bubbles near a boundary blow-up point are entire finite-energy solutions of the super-Liouville equation on $\R^2$. Second, we prove that the energy $\int_{B_r^+}\left(2e^{2 u_n}-e^{u_n}|\psi_n|^2\right)\,dx$ is greater than the energy of a single bubble ($4\pi$). We remark that this is trivial for the Liouville equation, since its energy density $2e^{2u_n}$ is positive. For the super-Liouville equation, however, the energy density  $2e^{2u_n}-e^{u_n}|\psi_n|^2$ does not have a fixed sign. Here, following the reduction argument of Ding-Tian \cite{DingWeiyueandTiangang} and the annular estimate in \cite[Lemma~3.1]{Jost-Zhou-Zhu-2014}, we establish the following energy identity for the spinor under a Dirichlet-chiral boundary condition
	$$
	\lim_{\delta\to0}
	\lim_{R\to+\infty}
	\limsup_{n\to\infty}
	\int_{\mathcal N_{\delta,R,n}}
	|\psi_n|^4\,dx
	=0.
	$$
With the help of above spinor's energy identity, we can show \eqref{equat:07}. See Sec. \ref{sec:proof-of-theorem}.

	We point out that the above proof strategy is not essentially restricted
	to planar domains. The same argument can be adapted, with only notational
	modifications, to the corresponding super-Liouville system on a compact
	spin Riemann surface with smooth boundary. Indeed, near an interior point
	or a boundary point, one may use an isothermal coordinate chart or a
	boundary isothermal coordinate chart, respectively. By the conformal
	covariance of the super-Liouville system and the chirality boundary
	condition, the local problem is reduced to the Euclidean form considered
	above, with smooth transformed Dirichlet data. Consequently, the
	small-energy regularity estimates, the blow-up rescaling, the neck
	analysis and the Pohozaev argument carry over without essential changes.
	We formulate the result on a bounded domain in $\mathbb R^2$ only to
	simplify the notation.

\

\noindent\textbf{Notations:} Let $B_r(x_0)$ be the ball in $\mathbb{R}^2$ with radius $r$ centered at $x_0$. Let $\partial B_r(x_0)$ be the boundary
of $B_r(x_0)$. Denote
\begin{equation*}
  B_r^+(x_0):=\left\{
    x = (s,t)\in B_r(x_0):t>0
  \right\}, \ \ B_r^-(x_0):=\left\{
    x = (s,t)\in B_r(x_0):t<0
  \right\}
\end{equation*}
and
\begin{equation*}
  \partial^0B^+_r(x_0):=\left\{
    x = (s,t)\in \partial B^+_r(x_0):t=0
  \right\}, \  \partial^+B^+_r(x_0):=\left\{
    x = (s,t)\in \partial B^+_r(x_0):t>0
  \right\}.
\end{equation*}
Denote
\begin{equation*}
  \mathbb{R}^2_a :=\left\{
    x = (s,t)\in \mathbb{R}^2:t>-a
  \right\}
\end{equation*}
for some $a\ge 0$ and
\begin{equation*}
  \partial \mathbb{R}^2_a :=\left\{
    x = (s,t)\in \mathbb{R}^2:t=-a
  \right\}
\end{equation*}
For simplicity of notations, we always denote $B_1(0)$, $B_1^+(0)$, $B_1^-(0)$, $B_R(0)$, $B_R^+(0)$, $\mathbb{R}^2_0$ by $B$, $B^+$, $B^-$, $B_R$, $B^+_R$, $\mathbb{R}^2_+$ respectively.
We also denote $\bar{x} = (s,-t)\in\mathbb{R}^2$ for $x=(s,t)\in\mathbb{R}^2$.

\
	
	The rest of paper is organized as follows. In Section~2, we establish the small
	energy regularity result and several local estimates which will be used
	in the blow-up analysis. We also derive the Pohozaev type identity and
	prove the estimates for the corresponding boundary terms. In Section~3,
	we study the blow-up scale and the neck region, and then complete the
	proof of Theorem~\ref{thm:main}.

	\section{Some basic lemmas}
	
	\
	
	In this section, we prove several basic lemmas for the super-Liouville equation near the Dirichlet-chiral boundary, including a small-energy regularity lemma, a Pohozaev-type identity, and related results.

	\
	
 We first recall the following interior small energy regularity
	estimate.
	
	\begin{lem}[Lemma 4.4 in \cite{Jost-Wang-Zhou}]\label{lem:interior-small-energy-regul}
		Let $(u_n,\psi_n)$ be a
		sequence of smooth solutions to
		\[
		\begin{cases}
			-\Delta u_n
			=
			2e^{2u_n}-e^{u_n}|\psi_n|^2
			&\text{in } B_1,\\
			\slashed{D}\psi_n
			=
			-e^{u_n}\psi_n
			&\text{in } B_1.
		\end{cases}
		\]
		Assume that
		\[
		\int_{B_1}|\psi_n|^4\,dx<+\infty.
		\]
		Then there exists a constant $0<\epsilon_0<\pi$ such that if
		\[
		\int_{B_1}e^{2u_n}\,dx<\epsilon_0,
		\]
		there holds
		\[
		\|u_n^+\|_{L^\infty(B_{\frac 14})}+\|\psi_n\|_{L^\infty(B_{\frac 14})}\leq C.
		\]
	\end{lem}
	
\
	
	We next prove the corresponding boundary small energy regularity	estimate.	
	\begin{lem}\label{lem:small-energy-regul}
		Let $(u_n,\psi_n)$ be a sequence of smooth solutions of
		\begin{equation}\label{equat:02}
			\begin{cases}
				-\Delta u_n= 2e^{2u_n} - e^{u_n} |\psi_n|^2 \ \ &in\ \  B^+_1(0),\\
				\slashed{D}\psi_n= -e^{u_n}\psi_n \ \ &in\ \ \ B^+_1(0),\\
				u_n=\phi_n  \ \ &on\ \  \partial ^0B^+_1(0),\\
				\mathbf{B}\psi_n=\mathbf{B}\varphi_n  \ \ &on\ \  \partial ^0B^+_1(0),
			\end{cases}
		\end{equation} with
		\begin{align*}
			\int_{B^+}|\psi_n|^4dx+\|\mathbf{B}\varphi_n\|_{C^0(\partial^0B^+)}+ osc_{\partial^0B^+}\phi_n\leq C.
		\end{align*}
		Then there exists a positive constant $\epsilon_1$, such that if $$\int_{B^+}e^{2u_n}dx\leq\epsilon_1,$$ the following properties hold:
		\begin{itemize}
			\item[(1)] If
			$
			\sup_n\phi_n(0)<+\infty,
			$
			then
			$$
			\|u_n^+\|_{L^\infty(B^+_{\frac12})}
			+
			\|\psi_n\|_{L^\infty(B^+_{\frac12})}
			\leq C.
			$$
			
			\item[(2)] If $\lim_{n\to\infty}\phi_n(0)=-\infty$, then $u_n\to -\infty$ uniformly in $B^+_{\frac{1}{2}}$ as $n\to\infty$.
		\end{itemize}
		
	\end{lem}
	\begin{proof}
		Extend $\phi_n$ continuously to $\partial B^+$ such that $$ osc_{\partial B^+}\phi_n\leq osc_{\partial^0 B^+}\phi_n\leq C.$$ Since $\sup_n\phi_n(0)<+\infty$, we have
		$\phi_n(0)\leq C$.
		
		\
		
		\noindent\textbf{Step 1.} Estimate for $\|\psi_n\|_{L^{16}(B^+_{\frac{3}{4}})}$.
		
		\
		
		Take a cut-off function $\eta\in C^\infty_0(B_1)$ such that $0\leq\eta\leq 1$, $\eta|_{B_{\frac{3}{4}}}\equiv 1$ and $\|\nabla\eta\|_{L^\infty}\leq C$. By the standard elliptic estimates of Dirac operator, for any $1<p<2$, we have
		\begin{align*}
			\|\eta\psi_n\|_{W^{1,p}(B^+)}&\leq C(\|\slashed{D}(\eta\psi_n)\|_{L^p(B^+)}+\|\mathbf{B}(\eta\psi_n)\|_{L^p(\partial B^+)})\\
			&\leq C(\|e^{u_n}\eta\psi_n\|_{L^p(B^+)}+\||\nabla\eta||\psi_n|\|_{L^p(B^+)}+\|\mathbf{B}(\eta\varphi_n)\|_{L^p(\partial B^+)})\\
			&\leq C\|e^{u_n}\|_{L^2(B^+)}\|\eta\psi_n\|_{L^{\frac{2p}{2-p}}(B^+)}+C(\|\psi_n\|_{L^4(B^+)}+\|\mathbf{B}\varphi_n\|_{C^0(\partial^0 B^+)})\\
			&\leq C\sqrt{\epsilon_1}\|\eta\psi_n\|_{W^{1,p}(B^+)}+C(\|\psi_n\|_{L^4(B^+)}+\|\mathbf{B}\varphi_n\|_{C^0(\partial^0 B^+)}),
		\end{align*} where we used Sobolev embedding theory and Young's inequality. Now, take $p=\frac{16}{9}$ and $\epsilon_1$ small enough, we get $$\|\eta\psi_n\|_{W^{1,p}(B^+)}\leq  C(\|\psi_n\|_{L^4(B^+)}+\|\mathbf{B}\varphi_n\|_{C^0( \partial^0 B^+)}),$$ which implies $\|\psi_n\|_{L^{16}(B^+_{\frac{3}{4}})}\leq C(\|\psi_n\|_{L^4(B^+)}+\|\mathbf{B}\varphi_n\|_{C^0(\partial^0 B^+)}).$
		
		\
		
		\noindent\textbf{Step 2.} Estimate for $\|u_n^+\|_{L^\infty(B_{\frac{1}{2}})}$ and $\|\psi_n\|_{L^\infty(B_{\frac{1}{2}})}$ .
		
		\
		
		Let $u_n^1$ be the solution of
		\begin{align*}
			\begin{cases}
				-\Delta u_n^1=2e^{2u_n}-e^{u_n}|\psi_n|^2\ \ &in\ \ B^+,\\
				u^1_n=0\ \ &on \ \ \partial B^+
			\end{cases}
		\end{align*}
		Since $\|\Delta u_n^1\|_{L^1(B^+)}\leq C\sqrt{\epsilon_1},$ taking $\epsilon_1$ small, by Theorem 1 in \cite{Brezis}, we have $\int_{B^+}e^{8|u_n^1|}dx\leq C.$
		
		Let $u_n^2$ be the solution of
		\begin{align*}
			\begin{cases}
				-\Delta u_n^2=0\ \ &in\ \ B^+,\\
				u^2_n=\phi_n(x)-\phi_n(0)\ \ &on \ \ \partial B^+.
			\end{cases}
		\end{align*}
		By standard elliptic estimate, we know $\|u_n^2\|_{C^0(B^+)}\leq C$.
		
		Let $u_n^3=u_n-u_n^1-u_n^2$. Then it is easy to see that
		\begin{align*}
			\begin{cases}
				-\Delta u_n^3=0\ \ &in\ \ B^+,\\
				u^3_n=\phi_n(0)\ \ &on \ \ \partial^0 B^+,\\
				u^3_n=u_n-\phi_n(x)+\phi_n(0)\ \ &on \ \ \partial^+ B^+.
			\end{cases}
		\end{align*}
		Since $\phi_n(0)\leq C$, taking $C$ larger if necessary, we have
		$$
		u_n^3-C\leq0\quad\mbox{on}\quad\partial^0B^+.
		$$
		Noting that
		$$
		-\Delta(u_n^3-C)=0\quad\mbox{in}\quad B^+,
		$$
		by the local estimate at the boundary in
		\cite[Theorem~9.26]{GilbargTrudinger} and the standard interior estimate, we get
		\begin{align*}
			\|(u_n^3-C)^+\|_{L^\infty(B^+_{\frac34})}
			&\leq C\|(u_n^3-C)^+\|_{L^1(B^+)}\leq C\|(u_n^3)^+\|_{L^1(B^+)}\\
			&\leq C\left(
			\|u_n^+\|_{L^1(B^+)}
			+\|u_n^1\|_{L^1(B^+)}
			+\|u_n^2\|_{L^1(B^+)}
			\right)
			\leq C.
		\end{align*}
		Thus,
		$
		\|(u_n^3)^+\|_{L^\infty(B^+_{\frac34})}\leq C.
		$
		Then it is easy to see that $\|\Delta u_n^1\|_{L^2(B_{\frac{3}{4}})}\leq C$ which implies $\|u_n^1\|_{C^{0}(B^+_{\frac{3}{4}})}\leq C$. Thus, we arrived at $$\|u_n^+\|_{L^\infty(B^+_{\frac{3}{4}})}\leq C.$$
		
		Combining this with the fact $\|\psi_n\|_{L^{16}(B^+_{\frac{3}{4}})}\leq C$, by the standard elliptic estimates of Dirac operator and Sobolev embedding, we get $$\|\psi_n\|_{W^{1,4}(B^+_{\frac{5}{8}})}\leq C(\|\slashed{D}\psi_n\|_{L^{4}(B^+_{\frac{3}{4}})}+\|\psi_n\|_{L^4(B^+_{\frac{3}{4}})} +\|\mathbf{B}\psi_n\|_{L^4(\partial ^0B^+_{\frac{3}{4}})})\leq C,$$ and $\|\psi_n\|_{C^0(B^+_{\frac{5}{8}})}\leq C.$
		
		\
		
		\noindent\textbf{Step 3.} Conclusion of the lemma.
		
		\
		
		By the previous estimate, there exists a constant $C_0>0$, independent of $n$, such that $u_n^3\leq C_0$ in $B^+_{\frac34}$. Set $v_n:=C_0-u_n^3$. Then $v_n\geq0$ and $v_n$ is harmonic in $B^+_{\frac34}$. Moreover, since $u_n^3=\phi_n(0)$ on $\partial^0B^+_{\frac34}$, we have $v_n=C_0-\phi_n(0)$ on $\partial^0B^+_{\frac34}$.
		
		We now apply the boundary weak Harnack inequality \cite[Theorem~8.26]{GilbargTrudinger}, specifically its $p=1$ version. Fix a sufficiently small $r>0$ and let $x_0\in\partial^0B^+_{\frac23}$. Applying \cite[Theorem~8.26]{GilbargTrudinger} to $v_n$ in $B_{2r}(x_0)\cap\mathbb R^2_+$, the boundary infimum appearing in that theorem is
		$$
		m=\inf_{\partial\mathbb R^2_+\cap B_{2r}(x_0)}v_n
		=C_0-\phi_n(0).
		$$
		As in Theorem~8.26, extend $v_n$ to $B_{2r}(x_0)$ by setting $\widetilde v_n=m$ in $B_{2r}(x_0)\setminus\mathbb R^2_+$. Since we use the $p=1$ case, the boundary weak Harnack inequality gives
		$$
		\fint_{B_{2r}(x_0)}\widetilde v_n
		\leq C\inf_{B_r(x_0)\cap\mathbb R^2_+}v_n.
		$$
		On the other hand, $\widetilde v_n=C_0-\phi_n(0)$ on $B_{2r}(x_0)\setminus\mathbb R^2_+$, and this set occupies a fixed positive proportion of $B_{2r}(x_0)$. Hence
		$$
		\fint_{B_{2r}(x_0)}\widetilde v_n
		\geq c\bigl(C_0-\phi_n(0)\bigr).
		$$
		Combining the two estimates, we obtain
		$$
		\inf_{B_r(x_0)\cap\mathbb R^2_+}v_n
		\geq c\bigl(C_0-\phi_n(0)\bigr),
		$$
		where $c>0$ is independent of $n$ and $x_0$.
		
		Covering the part of $B^+_{\frac12}$ near the flat boundary by finitely many such half-balls, and then using the standard interior Harnack inequality for the positive harmonic function $v_n$ along a finite chain of interior balls, we obtain
		$$
		\inf_{B^+_{\frac12}}v_n
		\geq c\bigl(C_0-\phi_n(0)\bigr).
		$$
		Since $v_n=C_0-u_n^3$, it follows that
		$$
		\sup_{B^+_{\frac12}}u_n^3
		\leq C+c\,\phi_n(0).
		$$
		Therefore, if $\phi_n(0)\to-\infty$, then $u_n^3\to-\infty$ uniformly in $B^+_{\frac12}$. Since $u_n^1$ and $u_n^2$ are uniformly bounded, we conclude that $u_n\to-\infty$ uniformly in $B^+_{\frac12}$.
	\end{proof}
	
	\
	
With the help of small energy regularity lemma and the interior blow-up theory for the super-Liouville equation developed in \cite{Jost-Wang-Zhou-Zhu-1}, we now show that there is no interior blow-up point.
	\begin{lem}\label{lem:compactness-away-blowup}
		Let $(u_n,\psi_n)$ be a sequence of smooth solutions of
		\eqref{equat:01} satisfying
		$$
		E(u_n,\psi_n;\Omega)\leq C.
		$$
		Define
		\begin{align*}
			\Sigma_1&:=\{x\in \overline{\Omega}\ |\ \mbox{there is a sequence of points } x_n\to x\mbox{ such that } u_n(x_n)\to+\infty\},\\
			\Sigma_2&:=\{x\in \overline{\Omega}\ |\ \mbox{there is a sequence of points } x_n\to x\mbox{ such that } |\psi_n(x_n)|\to+\infty\}.
		\end{align*}
		Then we have
		\begin{itemize}
			\item[(1)] $\Sigma=\Sigma_1\cup\Sigma_2$ is a finite point set and
			$\Sigma\cap\Omega=\emptyset$.
			
			\
			
			\item[(2)] $u_n$ and $\psi_n$ are bounded in
			$L^\infty_{loc}(\overline{\Omega}\setminus \Sigma)$. Passing to a
			subsequence, we have
			$(u_n,\psi_n)\to (u,\psi)$ in
			$C^2_{loc}(\overline{\Omega}\setminus \Sigma)$ as $n\to\infty$, where
			$(u,\psi)$ satisfy $$E(u,\psi;\Omega)+\|\nabla\psi\|_{L^{\frac{4}{3}}(\Omega)}<\infty$$ and
			\begin{equation*}
				\begin{cases}
					-\Delta u= 2e^{2u} - e^{u} |\psi|^2 \ \ &in\ \  \Omega,\\
					\slashed{D}\psi= -e^{u}\psi \ \ &in\ \ \ \Omega,\\
					u=u_0  \ \ &on\ \  \partial  \Omega\setminus \Sigma,\\
					\mathbf{B}\psi=0 \ \ &on\ \  \partial  \Omega\setminus \Sigma.
				\end{cases}
			\end{equation*}
			
		\end{itemize}
	\end{lem}
	
	\begin{proof}
We divide the proof into two parts.

\

\noindent \textbf{Step 1:} The first conclusion of the lemma.

\

			Define
	\[
	S_1
	:=
	\bigcap_{r>0}
	\left\{
	x\in\overline{\Omega}
	\ \middle|\
	\liminf_{n\to\infty}
	\int_{B_r(x)\cap\Omega}e^{2u_n}\,dx
	\geq\varepsilon_*
	\right\},
	\]
where $\varepsilon_*=\min\{\epsilon_0,\epsilon_1\}>0$ and $\epsilon_0,\epsilon_1$ are the constants in Lemma~\ref{lem:interior-small-energy-regul} and
	Lemma~\ref{lem:small-energy-regul}.

	We first claim that
	$
	S_1=\Sigma_1.
	$
	Indeed, for any $x_0\in\Sigma_1$, if $x_0\notin S_1$, then there exist
	$r>0$ and a subsequence, still denoted by $n$, such that
	$$
	\int_{B_r(x_0)\cap\Omega}e^{2u_n}\,dx<\varepsilon_*.
	$$
	By Lemma~\ref{lem:interior-small-energy-regul} and
	Lemma~\ref{lem:small-energy-regul}, after
	flattening the boundary when necessary, $u_n^+$ is uniformly
	bounded in $B_{\frac r2}(x_0)\cap\Omega$, which contradicts to the fact that
	$x_0\in\Sigma_1$.
	
	Conversely, for any $x_0\in S_1$, if $x_0\notin\Sigma_1$, then there exist $r_0>0$
	and $C>0$ such that
	$$
	u_n\leq C
	\quad\mbox{in}\quad
	B_{r_0}(x_0)\cap\Omega.
	$$
	Otherwise, after passing to a subsequence, one could find
	$x_n\to x_0$ such that $u_n(x_n)\to+\infty$. Taking
	$0<r<r_0$ sufficiently small, we have
	$$
	\int_{B_r(x_0)\cap\Omega}e^{2u_n}\,dx
	\leq
	e^{2C}|B_r(x_0)\cap\Omega|
	<\varepsilon_*,
	$$
	and hence $x_0\notin S_1$. Therefore, $S_1=\Sigma_1$.
	
	The energy bound implies that $S_1$ is finite. Indeed, for any
	distinct points $x_1,\ldots,x_N\in S_1$, choosing mutually
	disjoint small balls centered at these points, we obtain
	$$
	C
	\geq
	\liminf_{n\to\infty}\int_\Omega e^{2u_n}\,dx
	\geq
	N\varepsilon_*.
	$$
	Thus $\Sigma_1=S_1$ is finite.
	
	For any $x_0\in \Sigma_2\setminus \Sigma_1$, by the standard elliptic estimates for the Dirac operator and the bound of $\|\psi_n\|_{L^4(\Omega)}$, we can show that $ |\psi_n|$ is locally and uniformly bounded near $x_0$ which is a contradiction. Thus, $ \Sigma_2\subset\Sigma_1, $ and hence $ \Sigma=\Sigma_1 $ is finite.

		Let $w$ be the solution of
		\begin{align*}
			\begin{cases}
				-\Delta w=0\ \ &in\ \Omega,\\
				w=u_0\ \ &on\ \partial\Omega.
			\end{cases}
		\end{align*}
		Since $u_n-w=0$ on $\partial\Omega$ and $\|\Delta u_n\|_{L^1(\Omega)}\leq C$, by the standard elliptic
		estimate, for any $1<q<2$, we have
		$$
		\|u_n-w\|_{W^{1,q}(\Omega)}\leq C.
		$$
		By the Sobolev embedding
		$
		W^{1,\frac43}(\Omega)\hookrightarrow L^4(\Omega),
		$
		we get
		$
		\|u_n\|_{L^4(\Omega)}\leq C.
		$
		
		Now, we show that
		$
		\Sigma\cap\Omega=\emptyset.
		$
		Suppose that there exists
		$
		p\in\Sigma\cap\Omega.
		$
		Since $\Sigma=\Sigma_1$ is finite, we can choose $R>0$ sufficiently
		small such that
		$$
		\overline{B_{2R}(p)}\subset\Omega,
		\qquad
		B_{2R}(p)\cap\Sigma=\{p\}.
		$$
		Since $B_{2R}(p)$ is contained completely in the interior of $\Omega$,
		by the standard interior blow-up analysis for the super-Liouville
		equation (see \cite[Theorem~1.3]{Jost-Wang-Zhou-Zhu-1}), passing to a subsequence,
		we have
		$
		u_n\to-\infty
		$
		uniformly on compact subsets of
		$
		B_{2R}(p)\setminus\{p\}.
		$
		This contradicts to
		$
		\|u_n\|_{L^4(\Omega)}\leq C.
		$
		Therefore,
		$
		\Sigma\cap\Omega=\emptyset.
		$
		
\

\noindent \textbf{Step 2:} The second conclusion of the lemma.

\

Firstly, by the standard elliptic estimates for Dirac operator, we have
\begin{align*}
\|\psi_n\|_{W^{1,\frac{4}{3}}(\Omega)}\leq C(\|\slashed D\psi_n\|_{L^{\frac{4}{3}}(\Omega)}+\|\mathbf{B}\psi_n\|_{L^{\frac{4}{3}}(\partial\Omega)})\leq C.
\end{align*}
		For any		$
		x_0\in \overline{\Omega}\setminus\Sigma.
		$
		Then there exists $r_0>0$ such that $$\|u_n^+\|_{L^\infty(B_{2r_0}(x_0)\cap\Omega)}+\|\psi_n\|_{L^\infty(B_{2r_0}(x_0)\cap\Omega)}\leq C.$$		
		Combining this with
		$
		\|u_n\|_{L^4(\Omega)}+\|\psi_n\|_{L^4(\Omega)}\leq C,
		$
		and using the standard interior and boundary elliptic estimates, we obtain
		$$
		\|u_n\|_{C^3(B_{r_0}(x_0)\cap\Omega)}
		+
		\|\psi_n\|_{C^3(B_{r_0}(x_0)\cap\Omega)}
		\leq C.
		$$
		
		Thus, passing to a subsequence, we get
		$
		(u_n,\psi_n)\to(u,\psi)
		$
		in
		$
		C^2_{\mathrm{loc}}
		\left(
		\overline{\Omega}\setminus\Sigma
		\right),
		$
		where $(u,\psi)$ satisfies
		\begin{equation*}
			\begin{cases}
				-\Delta u=2e^{2u}-e^u|\psi|^2
				&in\ \Omega,\\
				\slashed{D}\psi=-e^u\psi
				&in\ \Omega,\\
				u=u_0
				&on\ \partial\Omega\setminus\Sigma,\\
				\mathbf{B}\psi=0
				&on\ \partial\Omega\setminus\Sigma.
			\end{cases}
		\end{equation*}
		Finally, by Fatou's lemma, we have
		$
		E(u,\psi;\Omega)+\|\nabla\psi\|_{L^{\frac{4}{3}}(\Omega)}<\infty.
		$
		This proves the lemma.
	\end{proof}
	
	\
	
	Next, we will show that the blow-up set $\Sigma=\emptyset$. Without loss of generality, we assume $0\in\Sigma\subset\partial \Omega$  and $B^+_1(0)\subset \Omega$ with $\partial^0B^+_1(0)\subset \partial \Omega$. Otherwise, for $p\in\Sigma$, using the isothermal coordinates at $p$, i.e. there exists a conformal map $f:\overline{B^+_1(0)}\to \overline{B_{r_0}(p)}\cap\overline{\Omega}$ such that $$f^*((dx^1)^2+(dx^2)^2)=e^{2\xi(x)}((dx^1)^2+(dx^2)^2), f(\partial^0B^+)\subset\partial \Omega,$$ we just need to consider its conformal transformation that $$\tilde{u}_n(x)=u_n\circ f+\xi(x),\ \ \tilde{\psi}_n(x)=e^{\frac{1}{2}\xi(x)}\psi_n\circ f$$ which satisfies
	\begin{equation*}
		\begin{cases}
			-\Delta \tilde{u}_n= 2e^{2\tilde{u}_n} - e^{\tilde{u}_n} |\tilde{\psi}_n|^2 \ \ &in\ \  B^+_1(0),\\
			\slashed{D}\tilde{\psi}_n= -e^{\tilde{u}_n}\tilde{\psi}_n \ \ &in\ \ \ B^+_1(0),\\
			\tilde{u}_n=u_0\circ f+\xi  \ \ &on\ \  \partial ^0B^+_1(0),\\
			\mathbf{B}\tilde{\psi}_n=0  \ \ &on\ \  \partial ^0B^+_1(0).
		\end{cases}
	\end{equation*}

	\
	
Now we derive the following Pohozaev type identity near a Dirichlet-chiral boundary which will be used in our later proof.
	\begin{lem}\label{lem:boundary-pohozaev-identity}
		Let $(u_n,\psi_n)$ be a sequence of solutions of
		\begin{equation}\label{eq:local-super-liouville-dirichlet}
			\begin{cases}
				-\Delta u_n
				=
				2e^{2u_n}-e^{u_n}|\psi_n|^2
				&\mbox{in } B_1^+(0),\\
				\slashed D\psi_n
				=
				-e^{u_n}\psi_n
				&\mbox{in } B_1^+(0),\\
				u_n=u_0
				&\mbox{on } \partial^0B_1^+(0),\\
				\mathbf B\psi_n=0
				&\mbox{on } \partial^0B_1^+(0).
			\end{cases}
		\end{equation}
		Denote $x=(s,t)$ and $\bar x=(s,-t)$. Then, for every $0<r<1$,
		the following identity holds
		\begin{align}\label{eq:halfdisk-pohozaev-nonzero-dirichlet}
			&\int_{B_r^+}
			\left(
			2e^{2 u_n}
			-
			e^{ u_n}|\psi_n|^2
			\right)\,dx
			\notag\\
			&\quad =
			r\int_{\partial^+B_r^+}
			\left(
			\left|\frac{\partial u_n}{\partial\nu}\right|^2
			-\frac12|\nabla  u_n|^2
			+e^{2  u_n}
			\right)\,d\sigma
			-\frac14
			\int_{\partial^+B_r^+}
			\left\langle
			\frac{\partial \psi_n}{\partial\nu},
			(x+\bar x)\cdot \psi_n
			\right\rangle\,d\sigma
			\notag\\
			&\qquad
			-\frac14
			\int_{\partial^+B_r^+}
			\left\langle
			(x+\bar x)\cdot \psi_n,
			\frac{\partial \psi_n}{\partial\nu}
			\right\rangle\,d\sigma
			+\int_{\partial^0B_r^+}
			s\,\frac{\partial u_0}{\partial s}(s,0)\,
			\frac{\partial  u_n}{\partial\nu}(s,0)\,ds .
		\end{align}
		Here $\nu$ denotes the outward unit normal vector field. On $\partial^+B_r^+$,
		$\nu$ is the radial outward normal, while on $\partial^0B_r^+$,
		$\nu=-\partial_t$.
	\end{lem}
	\begin{proof}
		Multiplying the first equation in \eqref{eq:local-super-liouville-dirichlet} by
		$x\cdot\nabla u_n$ and integrating over $B_r^+$, we get
		$$
		-\int_{B_r^+}\Delta u_n\,x\cdot\nabla u_n\,dx
		=
		\int_{B_r^+}
		\left(
		2e^{2u_n}-e^{u_n}|\psi_n|^2
		\right)
		x\cdot\nabla u_n\,dx .
		$$
		Integration by parts, we get
		\begin{align}
			-\int_{B_r^+}\Delta u_n\,x\cdot\nabla u_n\,dx	&=
			-\int_{\partial B_r^+}
			\frac{\partial u_n}{\partial\nu}
			x\cdot\nabla u_n\,d\sigma+
			\int_{B_r^+}
			\nabla u_n\cdot\nabla(x\cdot\nabla u_n)\,dx \notag\\
&=-\int_{\partial B_r^+}
			\frac{\partial u_n}{\partial\nu}
			x\cdot\nabla u_n\,d\sigma+
			\int_{B_r^+}
			|\nabla u_n|^2\,dx+	\frac{1}{2}\int_{B_r^+}x\cdot\nabla|\nabla u_n|^2\,dx \notag\\
&=-\int_{\partial B_r^+}
			\frac{\partial u_n}{\partial\nu}
			x\cdot\nabla u_n\,d\sigma+	\frac{1}{2}\int_{\partial B_r^+}|\nabla u_n|^2 x\cdot\nu\,dx \notag\\
&=			-r\int_{\partial^+B_r^+}
			\left(
			\left|\frac{\partial u_n}{\partial\nu}\right|^2
			-\frac12|\nabla u_n|^2
			\right)d\sigma-
			\int_{\partial^0B_r^+}
			s\frac{\partial u_0}{\partial s}(s,0)
			\frac{\partial u_n}{\partial\nu}(s,0)\,ds,
			\label{eq:pohozaev-scalar-left}
		\end{align}
		and
		\begin{align}
\int_{B_r^+}
			2e^{2u_n}x\cdot\nabla u_n\,dx
			&=
			r\int_{\partial^+B_r^+}e^{2u_n}\,d\sigma
			-
			2\int_{B_r^+}e^{2u_n}\,dx ,
			\label{eq:pohozaev-e2u}\\
			\int_{B_r^+}
			e^{u_n}|\psi_n|^2x\cdot\nabla u_n\,dx
			&=
			r\int_{\partial^+B_r^+}
			e^{u_n}|\psi_n|^2\,d\sigma-
			\int_{B_r^+}
			e^{u_n}x\cdot\nabla(|\psi_n|^2)\,dx
			-
			2\int_{B_r^+}
			e^{u_n}|\psi_n|^2\,dx .
			\label{eq:pohozaev-spinor-density}
		\end{align}
		Combining \eqref{eq:pohozaev-scalar-left},
		\eqref{eq:pohozaev-e2u} and
		\eqref{eq:pohozaev-spinor-density}, we get
		\begin{align}
			r\int_{\partial^+B_r^+}
			\left(
			\left|\frac{\partial u_n}{\partial\nu}\right|^2
			-\frac12|\nabla u_n|^2
			\right)d\sigma=&
			2\int_{B_r^+}\left(e^{2u_n}-e^{u_n}|\psi_n|^2\right)\,dx
			-
			r\int_{\partial^+B_r^+}\left(e^{2u_n}-e^{u_n}|\psi_n|^2\right)\,d\sigma\notag
			\\
			&-
			\int_{B_r^+}
			e^{u_n}x\cdot\nabla(|\psi_n|^2)\,dx-
			\int_{\partial^0B_r^+}
			s\frac{\partial u_0}{\partial s}(s,0)
			\frac{\partial u_n}{\partial\nu}(s,0)\,ds .
			\label{eq:pohozaev-before-spinor}
		\end{align}
		
		Next, we deal with the spinor terms. By the chirality boundary
		condition, we extend $\psi_n$ to $B_r$ by
		$$
		\widehat\psi_n(x)
		=
		\begin{cases}
			\psi_n(x),&x\in B_r^+,\\[1mm]
			ie_1\cdot\psi_n(\bar x),&x\in B_r^-.
		\end{cases}
		$$
		Setting
		$$
		\widehat A_n(x)
		=
		\begin{cases}
			e^{u_n(x)},&x\in B_r^+,\\[1mm]
			e^{u_n(\bar x)},&x\in B_r^-,
		\end{cases}
		$$
		By Lemma 3.4 in \cite{Jost-Zhou-Zhu-2014}, we have
		$$
		-\slashed D\widehat\psi_n
		=
		\widehat A_n\widehat\psi_n
		\quad\mbox{in}\quad B_r.
		$$
		By the Schr\"odinger--Lichnerowicz formula
		$
		\slashed D^2=-\Delta,
		$
		we have
		$$
		-\Delta\widehat\psi_n
		=
		-\slashed D\widehat A_n\cdot\widehat\psi_n
		+
		\widehat A_n^2\widehat\psi_n=-\nabla_{e_\alpha}\widehat A_n e_\alpha\cdot \widehat \psi_n +
		\widehat A_n^2\widehat\psi_n.
		$$
		Multiplying this equation by $x\cdot\widehat\psi_n$, taking
		the Hermitian conjugate identity, and integrating by parts, by a similar computation as in the Proposition 2.7 of \cite{Jost-Wang-Zhou-Zhu-1}, we get
\begin{align*}
			&r\int_{\partial B_r}
			\widehat A_n |\widehat\psi_n|^2\,d\sigma
			-
			\int_{B_r}
			\widehat A_n x\cdot\nabla(|\widehat\psi_n|^2)\,dx
			\notag\\
			&\quad=
			\frac12
			\int_{\partial B_r }
			\left\langle
			\frac{\partial \widehat\psi_n}{\partial\nu},
			x\cdot \widehat\psi_n
			\right\rangle\,d\sigma+
			\frac12
			\int_{\partial B_r }
			\left\langle
			x\cdot \widehat\psi_n,
			\frac{\partial\psi_n}{\partial\nu}
			\right\rangle\,d\sigma
			+
			\int_{B_r }
			\widehat A_n|\widehat\psi_n|^2\,dx .
		\end{align*}

		By the definition of the reflection,		 we obtain
		\begin{align}
			&r\int_{\partial^+B_r^+}
			e^{u_n}|\psi_n|^2\,d\sigma
			-
			\int_{B_r^+}
			e^{u_n}x\cdot\nabla(|\psi_n|^2)\,dx
			\notag\\
			&\quad=
			\frac14
			\int_{\partial^+B_r^+}
			\left\langle
			\frac{\partial\psi_n}{\partial\nu},
			(x+\bar x)\cdot\psi_n
			\right\rangle\,d\sigma+
			\frac14
			\int_{\partial^+B_r^+}
			\left\langle
			(x+\bar x)\cdot\psi_n,
			\frac{\partial\psi_n}{\partial\nu}
			\right\rangle\,d\sigma
			+
			\int_{B_r^+}
			e^{u_n}|\psi_n|^2\,dx .
			\label{eq:pohozaev-spinor-part}
		\end{align}
		Substituting \eqref{eq:pohozaev-spinor-part} into
		\eqref{eq:pohozaev-before-spinor}, we get
		\eqref{eq:halfdisk-pohozaev-nonzero-dirichlet}.
	\end{proof}
	
	\
	
In order to estimate the energies in the boundary neck domain, we need the following lemma.
	\begin{lem}
		\label{lem:boundary-annular-spinor-estimate}
		Let  $(u_n,\psi_n)$ be a sequence of smooth solutions to
		$$
		\slashed D\psi_n=-e^{u_n}\psi_n
		\quad\mbox{in } B_1^+,
		\qquad
		\mathbf B\psi_n=0
		\quad\mbox{on } \partial^0B_1^+.
		$$
Denote $
		A_{r_1,r_2}^+
		:=
		\left\{
		x\in\mathbb R^2_+:
		r_1\leq|x|\leq r_2
		\right\}
		$
		where
		$
		0<r_1<2r_1<\frac{r_2}{2}<r_2.
		$
		Then we have
		\begin{align*}
			&\left(
			\int_{A_{2r_1,\frac{r_2}{2}}^+}
			|\nabla\psi_n|^{\frac43}\,dx
			\right)^{\frac34}
			+
			\left(
			\int_{A_{2r_1,\frac{r_2}{2}}^+}
			|\psi_n|^4\,dx
			\right)^{\frac14}
			\notag\\
			&\quad\leq
			C_0
			\left(
			\int_{A_{r_1,r_2}^+}
			e^{2u_n}\,dx
			\right)^{\frac12}
			\left(
			\int_{A_{r_1,r_2}^+}
			|\psi_n|^4\,dx
			\right)^{\frac14}
			+
			C
			\left(
			\int_{A_{r_1,2r_1}^+}
			|\psi_n|^4\,dx
			\right)^{\frac14}
			+
			C
			\left(
			\int_{A_{\frac{r_2}{2},r_2}^+}
			|\psi_n|^4\,dx
			\right)^{\frac14},
		\end{align*} where  $C_0>0$ and $C>0$ are two constants independent of		$n,r_1,r_2$.
	\end{lem}
	\begin{proof}
	For $0<a<b$, set
	$$
	A_{a,b}
	:=
	\left\{
	x\in\mathbb R^2:
	a\leq |x|\leq b
	\right\}.
	$$
	For $x=(x_1,x_2)$, denote $\bar x=(x_1,-x_2)$. By the zero chirality boundary condition, we
	extend $\psi_n$ and $e^{u_n}$ to $B_1$ by setting
	$$
	\widehat\psi_n(x)
	=
	\begin{cases}
		\psi_n(x),
		&x\in B_1^+,\\[1mm]
		ie_1\cdot\psi_n(\bar x),
		&x\in B_1^-,
	\end{cases}
	\qquad
	\widehat A_n(x)
	=
	\begin{cases}
		e^{u_n(x)},
		&x\in B_1^+,\\[1mm]
		e^{u_n(\bar x)},
		&x\in B_1^-.
	\end{cases}
	$$
	Then
	$$
	\slashed D\widehat\psi_n
	=
	-\widehat A_n\widehat\psi_n
	\quad\mbox{in}\quad B_1.
	$$

	By	Lemma 3.1 in \cite{Jost-Wang-Zhou-Zhu-1}, we get
	\begin{align*}
		&\left(
		\int_{A_{2r_1,\frac{r_2}{2}}}
		|\nabla\widehat\psi_n|^{\frac43}\,dx
		\right)^{\frac34}
		+
		\left(
		\int_{A_{2r_1,\frac{r_2}{2}}}
		|\widehat\psi_n|^4\,dx
		\right)^{\frac14}
		\\
		&\quad\leq
		C_0
		\left(
		\int_{A_{r_1,r_2}}
		\widehat A_n^2\,dx
		\right)^{\frac12}
		\left(
		\int_{A_{r_1,r_2}}
		|\widehat\psi_n|^4\,dx
		\right)^{\frac14}+
		C
		\left(
		\int_{A_{r_1,2r_1}}
		|\widehat\psi_n|^4\,dx
		\right)^{\frac14}
		+
		C
		\left(
		\int_{A_{\frac{r_2}{2},r_2}}
		|\widehat\psi_n|^4\,dx
		\right)^{\frac14},
	\end{align*} where  $C_0>0$ and $C>0$ are two constants independent of		$n,r_1,r_2$.
	
	By the definition of the reflection, we obtain the conclusion of the lemma.
\end{proof}

	\
	
At the end of this section, we use potential analysis to get following two lemmas.
	\begin{lem}\label{lem:boundary-decay-estimates}
		Let $(u_n,\psi_n)$ be a sequence of solutions of \eqref{equat:01}. Suppose that, passing to a subsequence,
		$$(u_n,\psi_n)\to(u,\psi)\quad\mbox{in}\quad C^2_{\mathrm{loc}}(\overline{\Omega}\setminus\Sigma),$$
		and that $0\in\Sigma$ is the only blow-up point in $B_1^+(0)$. Then we have
		\begin{itemize}
			\item[(1)] For any $x\in B^+_{\frac{1}{2}}(0)\setminus\{0\}$, there holds
			$$
			u(x)+\ln|x|\leq C,\qquad
			\sqrt{|x|}|\psi(x)|\leq C,\qquad
			|x||\nabla u(x)|\leq C.
			$$
			
			\item[(2)] We have
			$$
			\lim_{r\to 0}\lim_{n\to\infty}
			\int_{\partial^0 B_r^+}|s||\nabla u_n|\,ds=0.
			$$
		\end{itemize}
	\end{lem}
	\begin{proof}
		For the first conclusion, since $e^{2u}\in L^1(B^+)$, we can choose $r_0>0$ sufficiently small such that
		$$
		\int_{B^+_{2r_0}}e^{2u}\,dx<\epsilon_1.
		$$
		For any $x\in B^+_{r_0}\setminus\{0\}$, by a standard rescaling on a ball or half ball with radius comparable to $|x|$ and the small energy regularity, we obtain
		$$
		u(x)+\ln|x|\leq C,\qquad
		\sqrt{|x|}|\psi(x)|\leq C.
		$$
		Moreover, by the standard elliptic estimates applied to
		$$
		-\Delta u=2e^{2u}-e^u|\psi|^2,
		$$
		after the same rescaling, we get
		$$
		|x||\nabla u(x)|\leq C.
		$$
		Since $(u,\psi)$ is smooth away from $0$, the same estimates hold in
		$B^+_{\frac12}\setminus B^+_{r_0}$. Therefore
		$$
		u(x)+\ln|x|\leq C,\qquad
		\sqrt{|x|}|\psi(x)|\leq C,\qquad
		|x||\nabla u(x)|\leq C
		$$
		for any $x\in B^+_{\frac12}(0)\setminus\{0\}$.
		
		For the second conclusion, let $w(x)$ be the solution of
		\begin{align*}
			\begin{cases}
				-\Delta w(x)=0\ \ &in\ \ \Omega,\\
				w(x)=u_0(x)\ \ &on\ \ \partial\Omega.
			\end{cases}
		\end{align*}
		By standard elliptic theory, we have
		$$
		\| w\|_{C^1(\overline{\Omega})}\leq C.
		$$
		Since $u_n-w=0$ on $\partial\Omega$, by the Green representation formula, we have
		\begin{equation}\label{equat:03}
			u_n(x)-w(x)
			=
			\int_\Omega G(x,y)
			\left(
			2e^{2u_n(y)}-e^{u_n(y)}|\psi_n(y)|^2
			\right)dy.
		\end{equation}
		where $G(x,y)$ is the Green function of $\Omega$ with the Dirichlet boundary condition. By H\"older's inequality and the energy bound,
		$$
		\int_\Omega
		\left|2e^{2u_n}-e^{u_n}|\psi_n|^2\right|dy
		\leq C.
		$$
		
		For any fixed $y\in\Omega$, since $G(x,y)>0$ in $\Omega$ and $G(x,y)=0$ on $\partial\Omega$, by Hopf's lemma, we have
		$$
		-\frac{\partial G}{\partial\nu_x}(x,y)\geq0,
		\qquad x\in\partial\Omega.
		$$
		On the other hand, applying the Green representation formula to the harmonic function $1$, we get
		$$
		1=
		-\int_{\partial\Omega}
		\frac{\partial G}{\partial\nu_x}(x,y)\,ds_x.
		$$
		Differentiating the Green representation formula on $\partial\Omega$, we obtain
		\begin{align*}
			\int_{\partial^0B_r^+}
			\left|
			\frac{\partial (u_n-w)}{\partial\nu}
			\right|ds
			&\leq
			\int_\Omega
			\left|2e^{2u_n(y)}-e^{u_n(y)}|\psi_n(y)|^2\right|
			\int_{\partial^0B_r^+}
			-\frac{\partial G}{\partial\nu_x}(x,y)\,ds_xdy\\
			&\leq
			\int_\Omega
			\left|2e^{2u_n(y)}-e^{u_n(y)}|\psi_n(y)|^2\right|dy\leq C.
		\end{align*}
		Hence
		$$
		\int_{\partial^0B_r^+}
		\left|\frac{\partial u_n}{\partial\nu}\right|ds\leq C.
		$$
		Since $u_n-w=0$ on $\partial\Omega$, we also have
		$$
		\left|\frac{\partial u_n}{\partial s}\right|
		=
		\left|\frac{\partial w}{\partial s}\right|
		\leq C
		$$
		on $\partial^0B_r^+$. Therefore,
		\begin{align*}
			\int_{\partial^0B_r^+}|s||\nabla u_n|ds
			&\leq
			r\int_{\partial^0B_r^+}
			\left|\frac{\partial u_n}{\partial\nu}\right|ds
			+
			\int_{\partial^0B_r^+}|s|
			\left|\frac{\partial u_n}{\partial s}\right|ds\leq Cr.
		\end{align*}
		Letting $n\to\infty$ and then $r\to0$, we get the second conclusion.
	\end{proof}
	
	\
	
	\begin{lem}\label{lem:boundary-radial-gradient-vanishing}
		Under assumptions of Lemma \ref{lem:boundary-decay-estimates}, we have
		$$
		\lim_{r\to 0}\lim_{n\to\infty}
		r\int_{\partial^+B_r^+}|\nabla u_n|^2d\sigma=0.
		$$
	\end{lem}
	
	\begin{proof}
Without loss of generality, we assume $\Sigma=\{0\}$.
		For any $x\in\overline{\Omega}\setminus \{0\}$, on one hand, since $\|2e^{2u_n} - e^{u_n} |\psi_n|^2\|_{L^\infty(\Omega\setminus B_\delta(0))}\leq C$, by dominated convergence theorem, we have
		\begin{align*}
			\lim_{\delta\to 0}\lim_{n\to\infty}\int_{\Omega\setminus B_\delta(0)} G(x,y)\left(2e^{2u_n(y)} - e^{u_n(y)} |\psi_n(y)|^2\right)dy&= \lim_{\delta\to 0}\int_{\Omega\setminus B_\delta(0)} G(x,y)\left(2e^{2u(y)} - e^{u(y)} |\psi(y)|^2\right)dy\\
			&=\int_{\Omega} G(x,y)\left(2e^{2u(y)} - e^{u(y)} |\psi(y)|^2\right)dy.
		\end{align*}
		On the other hand, since $G(x,y)$ is continuous at $y=0$, for any $\epsilon>0$, there exists $\delta_0=\delta_0(\epsilon)>0$ such that when $y\in B^+_{\delta_0}(0)$, there holds $|G(x,y)|<\epsilon$. Thus, for any $\delta<\delta_0$, we get
		\begin{align*}
			\left|\int_{ B^+_\delta(0)} G(x,y)\left(2e^{2u_n(y)} - e^{u_n(y)} |\psi_n(y)|^2\right)dy\right|\leq C\epsilon,
		\end{align*} which implies
		\begin{align*}
			\lim_{\delta\to 0}\lim_{n\to\infty}\int_{B^+_\delta(0)} G(x,y)\left(2e^{2u_n(y)} - e^{u_n(y)} |\psi_n(y)|^2\right)dy=0.
		\end{align*}
		
		Now letting $n\to\infty$ in \eqref{equat:03}, for any $x\in\overline{\Omega}\setminus \{0\}$, we obtain
		\begin{align}\label{equat:08}
			u(x)-w(x)&=\lim_{n\to\infty}\int_\Omega G(x,y)\left(2e^{2u_n(y)} - e^{u_n(y)} |\psi_n(y)|^2\right)dy\notag\\
			&=\lim_{\delta\to 0}\lim_{n\to\infty}\int_{\Omega\setminus B_\delta(0)} G(x,y)\left(2e^{2u_n(y)} - e^{u_n(y)} |\psi_n(y)|^2\right)dy\notag\\&\quad + \lim_{\delta\to 0}\lim_{n\to\infty}\int_{ B_\delta(0)} G(x,y)\left(2e^{2u_n(y)} - e^{u_n(y)} |\psi_n(y)|^2\right)dy \notag\\
			&=\int_{\Omega} G(x,y)\left(2e^{2u(y)} - e^{u(y)} |\psi(y)|^2\right)dy.
		\end{align}
		
		For any $\epsilon\in (0,\frac{1}{100})$, take $\rho=\rho(\epsilon)>0$ small such that $$\int_{B^+_{\rho}} \left|2e^{2u(y)} - e^{u(y)} |\psi(y)|^2\right|dy<\epsilon^2.$$
		Thus, for any $x\in B^+_{\frac{1}{4}\rho}$, by \eqref{equat:08} and the first conclusion of Lemma \ref{lem:boundary-decay-estimates}, we have
		\begin{align}\label{inequ:01}
			&|\nabla u(x)|\notag\\&\leq \frac{1}{2\pi}\int_\Omega \frac{1}{|x-y|}\left|2e^{2u(y)} - e^{u(y)} |\psi(y)|^2\right|dy+O(1)\notag\\
			&=\frac{1}{2\pi}\int_{B_{\epsilon |x|}(x)} \frac{1}{|x-y|}\left|2e^{2u(y)} - e^{u(y)} |\psi(y)|^2\right|dy +\frac{1}{2\pi}\int_{B^+_\rho\setminus B_{\epsilon |x|}(x)} \frac{1}{|x-y|}\left|2e^{2u(y)} - e^{u(y)} |\psi(y)|^2\right|dy\notag\\&\quad  +\frac{1}{2\pi}\int_{\Omega\setminus B^+_\rho} \frac{1}{|x-y|}\left|2e^{2u(y)} - e^{u(y)} |\psi(y)|^2\right|dy+O(1)\notag\\
			&\leq \frac{C}{|x|^2}\int_{B_{\epsilon |x|}(x)} \frac{1}{|x-y|}dy+\frac{C}{\epsilon|x|}\int_{B^+_\rho} \left|2e^{2u(y)} - e^{u(y)} |\psi(y)|^2\right|dy+\frac{O(1)}{\rho}\notag\\
			&\leq C\left(\frac{\epsilon}{|x|}+\frac{1}{\rho}\right).
		\end{align}
		Then we have
		\begin{align*}
			\lim_{r\to 0}\lim_{n\to\infty}r\int_{\partial^+B_r^+}|\nabla u_n|^2d\sigma&=\lim_{r\to 0}r\int_{\partial^+B_r^+}|\nabla u|^2d\sigma\\
			&\leq C\lim_{r\to 0}r\int_{\partial^+B_r^+}\left(\frac{\epsilon}{|x|^2}+\frac{1}{\rho^2}\right)d\sigma\\
			&\leq C\epsilon,
		\end{align*} which immediately implies the conclusion of the lemma.
		
	\end{proof}

	\
	
	\section{Proof of Theorem \ref{thm:main}}\label{sec:proof-of-theorem}
	
	\

In this section, we will prove our main Theorem \ref{thm:main}.

\
	
	Let $x_n\in \overline{B^+}$ be the point such that $$u_n(x_n)=\sup_{B^+}u_n(x).$$ Since $0\in \Sigma$ is the only blow-up point in $\overline{B^+}$, it is easy to check that $$x_n\to 0\ \ and \ \ u_n(x_n)\to +\infty\ \ as \ \ n\to\infty.$$

	Set $\lambda_n:=e^{ -u_n(x_n)}$ and $$v_n(x):=u_n(x_n+\lambda_nx)+\log\lambda_n,\ \ \zeta_n(x):=\sqrt{\lambda_n}\psi_n(x_n+\lambda_nx).$$ Denote $$\Omega_n:=\{x\in\R^2 |x_n+\lambda_nx\in B^+\}$$ and $$d_n:=dist(x_n,\partial^0B^+).$$
	
	We first claim the following lemma.
	\begin{lem}\label{lem:distance-to-boundary-scale}
		Passing to a subsequence, we have
		$$
		\lim_{n\to\infty}\frac{d_n}{\lambda_n}=+\infty.
		$$
	\end{lem}
	
	\begin{proof}
		We prove the lemma by a contradiction argument. Otherwise, there exists $R\in [0,+\infty)$ such that $\lim_{n\to\infty}\frac{d_n}{\lambda_n}=R$. Denote $x_n=(s_n,t_n)$ and $\tilde{x}_n=(0,-\frac{d_n}{\lambda_n})=(0,-\frac{t_n}{\lambda_n})$. Setting $$\tilde{v}_n(x):=v_n(x+\tilde{x}_n),\ \ \tilde{\zeta}_n(x):=\zeta_n(x+\tilde{x}_n),$$ we can easily check that $\tilde{v}_n(x)\leq 0$, $\tilde{v}_n(-\tilde{x}_n)= 0$, $\tilde{v}_n(0)=u_0(s_n,0)+\ln\lambda_n$ and
		\begin{equation}
			\begin{cases}
				-\Delta \tilde{v}_n= 2e^{2\tilde{v}_n} - e^{\tilde{v}_n} |\tilde{\zeta}_n|^2 \ \ &in\ \  B^+_{\frac{1}{2}\lambda_n^{-1}}(0),\\
				\slashed{D}\tilde{\zeta}_n= -e^{\tilde{v}_n}\tilde{\zeta}_n \ \ &in\ \ \ B^+_{\frac{1}{2}\lambda_n^{-1}}(0),\\
				\tilde{v}_n(s,0)=u_0(s_n+\lambda_ns,0)+\ln\lambda_n  \ \ &on\ \  \partial ^0B^+_{\frac{1}{2}\lambda_n^{-1}}(0),\\
				\mathbf{B}\tilde{\zeta}_n=0  \ \ &on\ \  \partial ^0B^+_{\frac{1}{2}\lambda_n^{-1}}(0).
			\end{cases}
		\end{equation}
		Since $\lim_{n\to\infty}\tilde{v}_n(0)=-\infty$, by Lemma \ref{lem:small-energy-regul}, we get $\tilde{v}_n(x)\to -\infty$ uniformly in $B^+_{R+1}(0)$ as $n\to\infty$. This is a contradiction to $\tilde{v}_n(-\tilde{x}_n)= 0$. We proved this lemma.
	\end{proof}

	\
	
	\begin{prop}[Energy identity]\label{prop:neck-spinor-energy-vanishing}
		For $R>1$ and $0<\delta<1$, denote
		$$
		\mathcal N_{\delta,R,n}
		:=
		B^+\cap
		\left\{
		x:
		R\lambda_n\leq |x-x_n|\leq\delta
		\right\}.
		$$
		Suppose that $\mathcal N_{\delta,R,n}$ contains no further bubble. Then we have
		$$
		\lim_{\delta\to0}
		\lim_{R\to+\infty}
		\limsup_{n\to\infty}
		\int_{\mathcal N_{\delta,R,n}}
		|\psi_n|^4\,dx
		=0.
		$$
	\end{prop}
	\begin{proof}
		For convenience, we write
		$$
		B_r^+(x_n):=B_r(x_n)\cap B^+.
		$$
		We divide the proof into following two steps.
		
		\
		
		\noindent\textbf{Step 1:}	We  claim that for $\forall \epsilon>0$, there exist $R_0=R_0(\epsilon)>100$, $\delta_0=\delta_0(\epsilon)<\frac{1}{100}$ such that for any $R>R_0$ and $\delta<\delta_0$, when $n$
		is big enough, we have
		\begin{equation}\label{eq:neck-unit-log-small}
			\int_{B_{2r}^+(x_n)\setminus B_{r}^+(x_n)}
			\left(
			e^{2u_n}+|\psi_n|^4
			\right)\,dx
			<\epsilon,\ \ \forall\ \lambda_nR\leq r\leq \delta.
		\end{equation}

		\
		
		We first have the following two facts.
		
		\
		
		\noindent\textbf{Fact 1:} Fix $T>100$, we have \begin{equation}\label{eq:neck-outer-end-small}
			\lim_{\delta\to 0}\lim_{n\to\infty}
			\int_{B_\delta^+(x_n)\setminus B_{T^{-1}\delta }^+(x_n)}
			\left(
			e^{2u_n}+|\psi_n|^4
			\right)\,dx
			=0.
		\end{equation}
		
		\
		
		In fact, since 	$x_n\to 0$, for $\forall \delta>0$, when $n$ is sufficiently
		large, we have
		$$
		B_\delta^+(x_n)\setminus B_{T^{-1}\delta}^+(x_n)
		\subset
		B_{2\delta}^+(0)\setminus
		B_{\frac{1}{2}T^{-1}\delta}(0).
		$$
		Noting that
		$$
		(u_n,\psi_n)\to(u,\psi)
		\quad\mbox{in}\quad
		C^2_{\mathrm{loc}}(\overline{B^+}\setminus\{0\}),
		$$
		and $(u,\psi)$ has finite energy, we obtain
		\begin{align*}
			\lim_{\delta\to0}
			\limsup_{n\to\infty}
			\int_{B_\delta^+(x_n)\setminus B_{T^{-1}\delta }^+(x_n)}
			\left(
			e^{2u_n}+|\psi_n|^4
			\right)\,dx
			=\lim_{\delta\to0}
			\int_{B_\delta^+(0)\setminus B_{T^{-1}\delta }^+(0)}
			\left(
			e^{2u}+|\psi|^4
			\right)\,dx=0.
		\end{align*}
		
		\
		
		\noindent\textbf{Fact 2:} Fix $T>100$, we have \begin{equation}\label{eq:neck-inner-end-small}
			\lim_{R\to \infty}\lim_{n\to\infty}
			\int_{B_{T\lambda_nR}^+(x_n)\setminus B_{\lambda_nR }^+(x_n)}
			\left(
			e^{2u_n}+|\psi_n|^4
			\right)\,dx
			=0.
		\end{equation}
		
		\	
		
		In fact, by Lemma \ref{lem:distance-to-boundary-scale}, we have
		$$
		\frac{d_n}{\lambda_n}\to+\infty.
		$$
		Then it is easy to see that, passing to a subsequence, we get
		$$
		(v_n,\zeta_n)\to(v,\zeta)
		\quad\mbox{in}\quad
		C^2_{\mathrm{loc}}(\mathbb R^2),
		$$
		where $(v,\zeta)$ is a nontrivial solution of the
		super-Liouville equation on $\mathbb R^2$ with  finite energy
		$$
		\int_{\mathbb R^2}
		\left(
		e^{2v}+|\zeta|^4
		\right)\,dx<+\infty.
		$$
		
		Therefore, we obtain
		\begin{align*}
			\lim_{R\to\infty}\lim_{n\to\infty}
			\int_{B^+_{T\lambda_n R}(x_n)\setminus B^+_{\lambda_n R}(x_n)}
			\left(
			e^{2u_n}+|\psi_n|^4
			\right)\,dx=	\lim_{R\to\infty}
			\int_{B_{TR}\setminus B_R}
			\left(
			e^{2v}+|\zeta|^4
			\right)\,dx=0.
		\end{align*}
		
		\
		
		We now prove \eqref{eq:neck-unit-log-small} by a contradiction argument.
		Otherwise, there exist $\varepsilon_0>0$ and $r_n$ such that
		$
		\lambda_nR\leq r_n\leq\delta
		$
		and
		\begin{equation}\label{eq:neck-bad-annulus}
			\int_{B_{2r_n}^+(x_n)\setminus B_{r_n}^+(x_n)}
			\left(
			e^{2u_n}+|\psi_n|^4
			\right)\,dx
			\geq\varepsilon_0.
		\end{equation}
		By \eqref{eq:neck-outer-end-small} and
		\eqref{eq:neck-inner-end-small}, we have
		$$
		\frac{\delta}{r_n}\to+\infty,
		\qquad
		\frac{r_n}{R\lambda_n}\to+\infty.
		$$
		In particular,
		$$
		r_n\to0,
		\qquad
		\frac{\lambda_n}{r_n}\to0.
		$$
		
		Set
		$$
		\widehat u_n(x)
		:=
		u_n(x_n+r_nx)+\log r_n,
		\qquad
		\widehat\psi_n(x)
		:=
		\sqrt{r_n}\psi_n(x_n+r_nx)
		$$ and $$\widehat{\Omega}_n:=\{x\in\R^2 |x_n+r_nx\in B^+\}.$$
		Then \eqref{eq:neck-bad-annulus} becomes
		\begin{equation}\label{eq:neck-bad-annulus-scaled}
			\int_{\widehat{\Omega}_n\cap
				(B_2\setminus B_{1})}
			\left(
			e^{2\widehat u_n}
			+
			|\widehat\psi_n|^4
			\right)\,dx
			\geq\varepsilon_0.
		\end{equation}
		Passing to a subsequence, we may assume that
		$$
		\frac{d_n}{r_n}\to\mu\in[0,+\infty].
		$$
		Denote $$\R^2_\mu:=\{(s,t)\in\R^2\ |\ t> -\mu\}.$$
		
		\
		
		We divide the discussion according to following three cases.

		\
		
		\noindent \textbf{Case 1:}  There exists $T>100$ such that $(\widehat{u}_n,\widehat{\psi}_n)$ has a energy concentration point in $\widehat\Omega_n\cap (B_T(0)\setminus B_{T^{-1}(0)})$.
		
		\
		
		In this case, by a standard rescaling argument, we will get the second bubble which is a contradiction.

		\
		
		\noindent \textbf{Case 2:}  $\widehat{u}_n\to -\infty$ uniformly in any compact subset of $\widehat\Omega_n\cap (\R^2\setminus \{0\})$.
		
		\
		
		In this case, we get that $\widehat{\psi}_n\to \widehat{\psi}$ in $C^1(K)$ for any compact subset $K\subset \widehat\Omega_n\cap (\R^2\setminus \{0\})$ where $\widehat{\psi}$ satisfies
		\begin{equation}
			\begin{cases}
				\slashed{D}\widehat{\psi}=0\ \ &in\ \ \R^2_\mu\setminus \{0\},\\
				\mathbf{B}\widehat{\psi}=0\ \ &on\ \ \partial \R^2_\mu\setminus \{0\}.
			\end{cases}
		\end{equation}
		
		If $\mu=+\infty$, by the removable singularity result for
		Dirac-harmonic maps in \cite{Chen-Jost-Li-Wang-2005}, it extends
		conformally to a harmonic spinor on $S^2$. Since there is no
		nontrivial harmonic spinor on $S^2$, the limit spinor is identically
		zero. Therefore,
		$$
		\int_{\widehat{\Omega}_n\cap
			(B_2\setminus B_{1})}
		|\widehat\psi_n|^4\,dx
		\to0.
		$$
		Moreover,
		$$
		\int_{\widehat{\Omega}_n\cap
			(B_2\setminus B_{1})}
		e^{2\widehat u_n}\,dx
		\to0.
		$$
		This contradicts \eqref{eq:neck-bad-annulus-scaled}.
		
		If $\mu\in [0,+\infty)$, after translation and
		reflection across the flat boundary, we obtain a finite-energy
		harmonic spinor on $\mathbb R^2\setminus\{0\}$ which implies $\widehat{\psi}\equiv 0$. Similarly, we will get a contradiction.
		
		\
		
		\noindent \textbf{Case 3:}  $(\widehat{u}_n,\widehat{\psi}_n)\to (\widehat{u},\widehat{\psi})$ in $C^2(K)$ for any compact subset  $K\subset\widehat\Omega_n\cap (\R^2\setminus \{0\})$.
		
		\
		
		In this case, we first claim that $\mu=+\infty$. Otherwise,  $d_n/r_n\to\mu<+\infty$. Noting that
		$$
		\widehat u_n(s,-d_n/r_n)
		=
		u_0(s_n+r_ns,0)+\log r_n
		\to-\infty,
		$$ by Lemma \ref{lem:small-energy-regul}, we get that  $\widehat{u}_n\to -\infty$ uniformly in any compact subset of $\widehat\Omega_n\cap (\R^2\setminus \{0\})$. This is a contradiction.
		
		Since now $\mu=+\infty$, similar to the argument as in
		\cite[Theorem~1.2]{Jost-Zhou-Zhu-2014} and \cite[Theorem~1.2]{Jost-Wang-Zhou-Zhu-1} for interior case, we can show that the singularity at $\{0\}$ is removable for $(\widehat{u},\widehat{\psi})$. Then we get the second bubble, which is also a contradiction. 	Thus \eqref{eq:neck-unit-log-small} is proved.
		
		\
		
		\noindent\textbf{Step 2:} Energy identity for the spinor.
		
		\
		
		We now apply Ding-Tian's reduction argument
		\cite{DingWeiyueandTiangang}, following the proof of
		\cite[Theorem~1.2]{Jost-Zhou-Zhu-2014}. Choose $\theta>0$
		sufficiently small such that
		$$
		C_0(2\theta)^{\frac12}<\frac14,
		$$
		where $C_0$ is larger than the constants in
		Lemma~\ref{lem:boundary-annular-spinor-estimate} and the interior
		annular estimate in
		\cite[Lemma~3.1]{Jost-Wang-Zhou-Zhu-1}.
		
		Fix $0<\epsilon<\frac{\theta}{6}$. By
		\eqref{eq:neck-unit-log-small}, taking $R$ sufficiently large,
		$\delta$ sufficiently small and $n$ sufficiently large, we have
		$$
		\int_{B_{2r}^+(x_n)\setminus B_r^+(x_n)}
		\left(e^{2u_n}+|\psi_n|^4\right)\,dx<\epsilon,
		\qquad
		R\lambda_n\leq r\leq\delta.
		$$
		Since $d_n/\lambda_n\to+\infty$ and $d_n\to0$, when $n$ is big
		enough, we can divide the neck domain as
		\begin{align*}
			\mathcal N_{\delta,R,n}
			={}&
			\left(
			B_\delta^+(x_n)\setminus
			B_{\frac{\delta}{2}}^+((s_n,0))
			\right)
			\cup
			\left(
			B_{\frac{\delta}{2}}^+((s_n,0))\setminus
			B_{2d_n}^+((s_n,0))
			\right)
			\\
			&\cup
			\left(
			B_{2d_n}^+((s_n,0))\setminus B_{d_n}(x_n)
			\right)
			\cup
			\left(
			B_{d_n}(x_n)\setminus B_{R\lambda_n}(x_n)
			\right).
		\end{align*}
		
		For the first and the third parts, noting that
		$$
		B_\delta^+(x_n)\setminus
		B_{\frac{\delta}{2}}^+((s_n,0))
		\subset
		B_\delta^+(x_n)\setminus B_{\frac{\delta}{4}}^+(x_n)
		$$
		and
		$$
		B_{2d_n}^+((s_n,0))\setminus B_{d_n}(x_n)
		\subset
		B_{4d_n}^+(x_n)\setminus B_{d_n}^+(x_n),
		$$
		by \eqref{eq:neck-unit-log-small}, we have $$\int_{B_\delta^+(x_n)\setminus
		B_{\frac{\delta}{2}}^+((s_n,0))}|\psi_n|^4dx+ \int_{B_{2d_n}^+((s_n,0))\setminus B_{d_n}(x_n)}|\psi_n|^4dx\leq C\epsilon.$$
		
		For the second part, for any
		$2d_n\leq r\leq\frac{\delta}{2}$, we have
		$$
		B_{2r}^+((s_n,0))\setminus B_r^+((s_n,0))
		\subset
		B_{3r}^+(x_n)\setminus B_{\frac r2}^+(x_n).
		$$
		Thus, by \eqref{eq:neck-unit-log-small}, the energy on every unit
		logarithmic annulus centered at $(s_n,0)$ is bounded by
		$3\epsilon$. Similar to \cite{DingWeiyueandTiangang}, we divide
		the second part into finitely many annular parts, all but possibly
		the last one having $e^{2u_n}$-energy equal to $\theta$. The number
		of these parts is bounded independently of $n$ by the uniform
		energy bound. Applying
		Lemma~\ref{lem:boundary-annular-spinor-estimate} to each annular
		part, we obtain
		$$
		\left(\int_A|\psi_n|^4\,dx\right)^{\frac14}
		\leq
		C_0(2\theta)^{\frac12}
		\left(\int_A|\psi_n|^4\,dx\right)^{\frac14}
		+C\epsilon^{\frac14}
		\leq
		\frac14
		\left(\int_A|\psi_n|^4\,dx\right)^{\frac14}
		+C\epsilon^{\frac14}.
		$$
		Hence the spinor energy on the second part is bounded by
		$C\epsilon$.
		
		Finally, since $B_{d_n}(x_n)\subset B^+$, applying the interior
		annular estimate in
		\cite[Lemma~3.1]{Jost-Wang-Zhou-Zhu-1} and the same reduction
		argument, we obtain
		$$
		\int_{B_{d_n}(x_n)\setminus B_{R\lambda_n}(x_n)}
		|\psi_n|^4\,dx
		\leq C\epsilon.
		$$
		Combining the above estimates, we get
		$$
		\int_{\mathcal N_{\delta,R,n}}|\psi_n|^4\,dx
		\leq C\epsilon.
		$$
		Therefore,
		$$
		\lim_{\delta\to0}
		\lim_{R\to+\infty}
		\limsup_{n\to\infty}
		\int_{\mathcal N_{\delta,R,n}}|\psi_n|^4\,dx
		=0.
		$$
		This proves the proposition.
	\end{proof}
	
	\
	
Now, we prove our main Theorem \ref{thm:main}.
	\begin{proof}[\textbf{Proof of Theorem~\ref{thm:main}}]
		By Lemma~\ref{lem:compactness-away-blowup}, $\Sigma$ is a finite point
		set and
		$
		\Sigma\cap\Omega=\emptyset.
		$
		Next we will show $\Sigma=\emptyset$.  If not, since $\Sigma$ is finite, without loss of generality, we assume that
		$\Sigma=\{0\}\subset\partial\Omega$ and  $\overline{B_1^+}\subset\overline{\Omega}$. By Lemma \ref{lem:compactness-away-blowup}, we know that passing to a subsequence,
		$$
		(u_n,\psi_n)\to(u,\psi)
		\quad\mbox{in}\quad
		C^2_{\mathrm{loc}}
		\left(
		\overline{B_1^+}\setminus\{0\}
		\right),
		$$
		with finite energy
		\begin{equation}\label{equat:04}		
			\int_{B^+}(e^{2u}+|\psi|^4+|\nabla\psi|^{\frac{4}{3}})\,dx\leq C.
		\end{equation}
		
		Let $x_n,\lambda_n,d_n,v_n$ and $\zeta_n$ be defined as above.
		By Lemma~\ref{lem:distance-to-boundary-scale}, we have
		$$
		\frac{d_n}{\lambda_n}\to+\infty.
		$$
		Therefore,  passing to a subsequence,
		$$
		(v_n,\zeta_n)\to(v,\zeta)
		\quad\mbox{in}\quad
		C^2_{\mathrm{loc}}(\mathbb R^2),
		$$
		where $(v,\zeta)$ is a nontrivial finite-energy solution of the
		super-Liouville equation on $\mathbb R^2$. By classification result in
		\cite[Proposition~6.3]{Jost-Wang-Zhou}, we have
		$$
		\int_{\mathbb R^2}
		\left(
		2e^{2v}-e^v|\zeta|^2
		\right)dx
		=
		4\pi.
		$$

		Without loss of generality, we assume there is only one bubble. By energy identity Proposition \ref{prop:neck-spinor-energy-vanishing}, we have
		\begin{align*}
			&		\lim_{n\to\infty}
			\int_{B_r^+(0)}
			\left(
			2e^{2u_n}-e^{u_n}|\psi_n|^2
			\right)dx\\ &=\lim_{\delta\to0}\lim_{n\to\infty}\int_{B_r^+(0)\setminus B^+_\delta(0)}	\left(	2e^{2u_n}-e^{u_n}|\psi_n|^2	\right)dx+  \lim_{R\to\infty}\lim_{\delta\to0}\lim_{n\to\infty}\int_{ B^+_\delta(0)\setminus B^+_{\lambda_nR}(x_n)}	\left(	2e^{2u_n}-e^{u_n}|\psi_n|^2	\right)dx  \\ &\quad +\lim_{R\to \infty}\lim_{n\to\infty}\int_{B^+_{\lambda_nR}(x_n)}	\left(	2e^{2u_n}-e^{u_n}|\psi_n|^2	\right)dx\\ &\geq\lim_{\delta\to0}\lim_{n\to\infty}\int_{B_r^+(0)\setminus B^+_\delta(0)}	\left(	2e^{2u_n}-e^{u_n}|\psi_n|^2	\right)dx+  \lim_{R\to\infty}\lim_{\delta\to0}\lim_{n\to\infty}\int_{ B^+_\delta(0)\setminus B^+_{\lambda_nR}(x_n)}	\left(	-e^{u_n}|\psi_n|^2	\right)dx  \\ &\quad +\lim_{R\to \infty}\lim_{n\to\infty}\int_{B^+_{\lambda_nR}(x_n)}	\left(	2e^{2u_n}-e^{u_n}|\psi_n|^2	\right)dx\\
			&=\int_{B_r^+(0)}	\left(	2e^{2u}-e^{u}|\psi|^2	\right)dx+ \int_{\R^2}	\left(	2e^{2v}-e^{v}|\zeta|^2	\right)dx
		\end{align*} which immediately implies
		\begin{equation}\label{equat:05}
			\lim_{r\to 0}\lim_{n\to\infty}
			\int_{B_r^+(0)}
			\left(
			2e^{2u_n}-e^{u_n}|\psi_n|^2
			\right)dx\geq 4\pi.
		\end{equation}

		By	Lemma~\ref{lem:boundary-pohozaev-identity}, we have
		\begin{align*}
			&\int_{B_r^+}
			\left(
			2e^{2 u_n}
			-
			e^{ u_n}|\psi_n|^2
			\right)\,dx
			\notag\\
			&\quad =
			r\int_{\partial^+B_r^+}
			\left(
			\left|\frac{\partial u_n}{\partial\nu}\right|^2
			-\frac12|\nabla  u_n|^2
			+e^{2  u_n}
			\right)\,d\sigma
			-\frac14
			\int_{\partial^+B_r^+}
			\left\langle
			\frac{\partial \psi_n}{\partial\nu},
			(x+\bar x)\cdot \psi_n
			\right\rangle\,d\sigma
			\notag\\
			&\qquad
			-\frac14
			\int_{\partial^+B_r^+}
			\left\langle
			(x+\bar x)\cdot \psi_n,
			\frac{\partial \psi_n}{\partial\nu}
			\right\rangle\,d\sigma
			+\int_{\partial^0B_r^+}
			s\,\frac{\partial u_0}{\partial s}(s,0)\,
			\frac{\partial  u_n}{\partial\nu}(s,0)\,ds .
		\end{align*}

		By \eqref{equat:04} and Fubini's theory, it is easy to see that there exists a sequence $r_j\to0$
		such that
		$$
		\begin{aligned}
			\lim_{j\to\infty}\lim_{n\to\infty}\bigg(r_j\int_{\partial^+B^+_{r_j}}e^{2u_n}d\sigma
			-\frac14
			\int_{\partial^+B_{r_j}^+}
			\left\langle
			\frac{\partial\psi_n}{\partial\nu},
			(x+\bar x)\cdot\psi_n
			\right\rangle\,d\sigma-\frac14
			\int_{\partial^+B_{r_j}^+}
			\left\langle
			(x+\bar x)\cdot\psi_n,
			\frac{\partial\psi_n}{\partial\nu}
			\right\rangle\,d\sigma\bigg)
			=0.
		\end{aligned}
		$$

		Combining this with  Lemma~\ref{lem:boundary-radial-gradient-vanishing} and Lemma~\ref{lem:boundary-decay-estimates},
		we get
		$$
		\lim_{j\to\infty}
		\lim_{n\to\infty}		
		\int_{B_{r_j}^+}
		\left(
		2e^{2u_n}-e^{u_n}|\psi_n|^2
		\right)dx		=		0.
		$$
		This contradicts \eqref{equat:05}.		Therefore,
		$$
		\Sigma=\emptyset.
		$$
		This proves Theorem~\ref{thm:main}.
	\end{proof}
	
	\providecommand{\bysame}{\leavevmode\hbox to3em{\hrulefill}\thinspace}
	\providecommand{\MR}{\relax\ifhmode\unskip\space\fi MR }
	\providecommand{\MRhref}[2]{%
		\href{http://www.ams.org/mathscinet-getitem?mr=#1}{#2}
	}
	\providecommand{\href}[2]{#2}

\end{document}